\documentclass[11pt]{amsart}

\usepackage{amsmath, amsfonts, amssymb,amsthm}
\usepackage{amstext}
\usepackage{mathrsfs}

\usepackage{float,epsf,subfigure}
\usepackage[all,cmtip]{xy}
\usepackage{hyperref}
\usepackage{algorithm}
\usepackage{algorithmic}
\usepackage{mathtools}
\usepackage{float}
\usepackage{bm}
\theoremstyle{plain}
\newtheorem{theorem}{Theorem}[section]

\newtheorem{cor}[theorem]{Corollary}
\newtheorem{def-thm}[theorem]{Definition-Theorem}
\newtheorem{lemma}[theorem]{Lemma}

\newtheorem*{tha}{Theorem A}
\newtheorem*{thb}{Theorem B}

\newtheorem*{thai}{Theorem I}
\newtheorem*{thaii}{Theorem II}
\newtheorem*{thaiii}{Theorem III}

\theoremstyle{definition}

\def\CC{\mathbb C}

{

\begin{document}
\title[Nevanlinna theory]{The Second Main Theorem for  spherically  symmetric K\"ahler manifolds}
\author[X.-J. Dong \& P.-C Hu] {Xianjing Dong \& Peichu Hu}

\address{School of Mathematics \\ China University of Mining and Technology \\  Jiangsu, Xuzhou, 221116, P. R. China}
\email{xjdong05@126.com}
\address{Department of Mathematics \\ Shandong University \\ Jinan \\ 250100 \\ Shandong \\ P. R. China}
\email{pchu@sdu.edu.cn}

  %\thanks{This work was partially supported by the
%NSFC (No. 11301260,11101201), the NSF of Jiangxi (No.
%20132BAB211003) and the YFED of Jiangxi (No. GJJ13078) of China.}

%\thanks{The second author was supported in part by the Simon Foundation (Grant No. 531604)}
\date{}
\maketitle \thispagestyle{empty} \setcounter{page}{1}

\begin{abstract}  
We investigate the value distribution of  holomorphic maps defined on one class of  K\"ahler manifolds.  With the very  natural settings, we establish a Second Main Theorem which is of the similar form as  ones of the classical Second Main Theorem for complex Euclidean spaces and complex  unit balls. 
\end{abstract}

\vskip\baselineskip

\setlength\arraycolsep{2pt}

\section{Introduction}

\subsection{Motivation}~

 Let $f$ be a nonconstant meromorphic function  on $\mathbb C^m$ or $\mathbb B^m$ and $a_1, \cdots, a_q$ be distinct points in $\overline\CC,$  where 
 $\mathbb B^m$ is the unit  ball (with  standard Euclidean metric) in $\mathbb C^m.$   We have the familiar  
 notations in Nevanlinna theory (see \cite{Nev}) such as  characteristic function $T_f(r),$ counting function $N_f(r, a),$   proximity function $m_f(r, a)$ and simple counting function $\overline N_f(r, a),$ etc., 
see  Noguchi \cite{noguchi}  and Ru \cite{ru}. 
We  recall the  Second Main Theorem in Nevanlinna theory for $\mathbb C^m$ and $\mathbb B^m$ as follows: 
 
$(i)$  \emph{In the case of $\mathbb C^m${\rm{:}} for every $\delta>0,$   we have  
$$(q-2)T_f(r)\leq \sum_{j=1}^q \overline N_f(r, a_j)+O\big(\log^+T_f(r)+\delta\log^+r\big)$$
holds for all $r\in(0,\infty)$ outside a set $E_\delta$ of finite Lebesgue measure.}

$(ii)$  \emph{In the case of $\mathbb B^m${\rm{:}} for every $\delta>0,$   we have  
$$(q-2)T_f(r)\leq \sum_{j=1}^q \overline N_f(r, a_j)+O(\log^+T_f(r)+\log\frac{1}{1-r})$$
holds for all $r\in(0,1)$ outside a set $E_\delta$ with $\int_{E_{\delta}}(1-r)^{-1}dr<\infty.$}

Note that  $T_f(r)$ is bounded from below by $O(\log r)$ as $r\rightarrow \infty.$ Hence, 
the result $(i)$ concludes the Little Picard Theorem  asserting that a nonconstant meromorphic function can omit at most two  points. 
This  result was extended  to the case of complex projective manifolds  by Carlson-Griffiths-King \cite{gri,gri1} under the dimension condition that 
the dimension of  target manifolds is not greater than $m.$  For the  $\mathbb B^m$-case, the Little Picard Theorem no longer holds, 
but one will find from   $(ii)$ that $f$ can omit at  most two points if $T_f(r)$ grows rapidly enough.

In 2010, by utilizing a technique of Brownian motion initialized by Carne \cite{carne},   Atsuji \cite {atsuji} established a Second Main Theorem of meromorphic functions on  a non-positively curved complete K\"ahler manifold. His theorem extends  the classical Nevanlinna theory for  $\CC^m$   (see $(i)$ in above), which says that 
\begin{tha}[Atsuji, \cite{atsuji}]   Let $M$ be  a complete K\"ahler manifold 
of non-positive sectional curvature and  $a_1,\cdots, a_q$ be distinct points in $\overline \CC.$
Let $f$ be a nonconstant meromorphic function on $M$. Then for every $\delta>0,$
  \begin{eqnarray*}
      (q-2)T_f(r)
     &\leq& \sum_{j=1}^q\overline{N}_f(r, a_j)+N(r, {\rm Ric}) +O\Big{(}\log^+ T_{f}(r)  +\log^+\big{(}G(r)\Phi(r)\big)\Big{)}
  \end{eqnarray*}
holds for all $r\in(0,\infty)$ outside a set $E_\delta$ of finite Lebesgue measure. 
\end{tha}
 In  Theorem A, the error terms such as $\log^+(G(r)\Phi(r))$ and $N(r, {\rm{Ric}})$ are involved, where $G$ is the solution of a certain second order ODE depending on the Green functions for geodesic balls in $M$ and $\Phi$ is expressed by $G,$ 
 and the curvature term $N(r, {\rm{Ric}})$ is   determined  by the Ricci curvature of $M$.  
 
Recently, the first named author \cite{dong}  investigated  Carlson-Griffiths theory \cite{gri} for complete K\"ahler manifolds by using the similar  probabilistic method. The  author    
 generalized Theorem A  by the following

\begin{thb}[Dong, \cite{dong}]  Let $M$ be  a complete K\"ahler manifold 
of non-positive sectional curvature and $V$ be a complex projective manifold satisfying that $\dim M\geq \dim V.$
  Let   $D\in|L|$  be a divisor of  simple normal crossing type, where $L$ is a holomorphic line bundle over $V$. Fix a Hermitian metric $\omega$ on $V.$ Let $f:M\rightarrow V$ be a differentiably non-degenerate meromorphic mapping.  Then for any $\delta>0,$
  \begin{eqnarray*}
      T_{f}(r, L)+T_{f}(r, K_V) 
     &\leq& \overline{N}_{f}(r, D)+O\Big{(}\log^+ T_{f}(r, \omega)-\kappa(r)r^2+\delta\log r\Big{)}
  \end{eqnarray*}
holds for all $r\in(0,\infty)$ outside a set $E_\delta$ of finite Lebesgue measure. 
\end{thb}

In Theorem B, the term $\kappa(r)$ is the minimal of the pointwise lower bound of the Ricci curvature for  geodesic balls in $M.$
Note from the above theorem, if one needs  to  receive a defect relation in Nevanlinna theory, then $T_f(r)$ must grow rapidly enough. 
 Both Theorem A and Theorem B cannot  conclude  $(ii),$ 
because  these estimate terms  
  are rough.  
   In order to establish a Second Main Theorem with  good error terms, some  Second Main Theorem with the form like $(i)$ or $(ii)$ is  expected.  
Thus, a natural question is that:  for what kind of  K\"ahler manifolds, the Second Main Theorem will be of  the similar  form as $(i)$ or $(ii)$?
 Motivated by that,  we  give  investigations to   Nevanlinna theory for a class of  K\"ahler manifolds,  i.e., the so-called \emph{spherically  symmetric K\"ahler manifolds} which satisfy the  requirements. 

\subsection{Main results}~

 Let $M_\sigma$ be a spherically  symmetric K\"ahler manifold of 
a pole $o$ and radius $R$ (see Section 2.2 for  definition). 
Consider a holomorphic map $f: M_\sigma\rightarrow N$ into a complex projective manifold  $N$ with $\dim N\leq \dim M_\sigma.$
In our settings, we  will remove all the restrictions  such as  completeness and non-positiveness of sectional curvature of a K\"ahler manifold in Theorems A and B. 
Without
going into to  the details of notations, we prove the following main result: 
\begin{thai}[=Theorem \ref{main theorem}]\label{} Let $M_\sigma$ be a spherically  symmetric K\"ahler manifold of complex dimension $m,$ with a pole $o$ and radius $R.$  Let $L$ be a positive line bundle over a complex projective manifold $N$ with $\dim_{\mathbb C}N\leq m,$  and  $D\in|L|$  be of  simple normal crossings.
 Let $f:M_\sigma\rightarrow N$ be a differentiably non-degenerate holomorphic map. Then 
 
$(a)$ For $R=\infty$   and  every $\delta>0,$  
  \begin{eqnarray*}
      T_{f}(r,L)+T_{f}(r,K_N)+T(r,\mathscr R_{M_\sigma}) 
     &\leq& \overline{N}_{f}(r,D)+O\big{(}\log^+ T_{f}(r,L)+\delta\log^+\sigma(r)\big{)}
  \end{eqnarray*}
holds for all $r\in(0,\infty)$ outside a set $E_\delta$ of finite Lebesgue measure. 

$(b)$ For $R<\infty$   and   every $\delta>0,$  
  \begin{eqnarray*}
     T_{f}(r,L)+T_{f}(r,K_N)+T(r,\mathscr R_{M_\sigma}) 
     &\leq& \overline{N}_{f}(r,D)+O\Big{(}\log^+ T_{f}(r,L)+\log\frac{1}{R-r}\Big{)}
  \end{eqnarray*}
 holds for all $r\in(0,R)$ outside a set $E_\delta$ with $\int_{E_{\delta}}(R-r)^{-1}dr<\infty.$
\end{thai}

We interpret  how  results $(i)$ and $(ii)$ can be derived  from ours. 
In the case that  $M_\sigma=\mathbb C^m$ (with standard Euclidean metric), we 
have  $R=\infty,  \sigma(r)=r$ 
and $\mathscr R_{M_\sigma}=0,$ where $\mathscr R_{M_\sigma}$ denotes the Ricci form of $M_\sigma.$ By $(a)$ in the above theorem,   
it immediately  deduces the theorem of Carlson-Griffiths-King (see Corollary \ref{cgk}).
The Second Main Theorem for the case that $M_\sigma= \mathbb B^m$ (with standard Euclidean metric) also  follows by noting that       
   $R=1, \sigma(r)=r$ and $\mathscr R_{\mathbb B^m}=0$ (see Corollary \ref{ccc}).  
 
A manifold is said to be \emph{non-parabolic} if it admits a non-constant positive superharmonic function,  and said to be \emph{parabolic} otherwise.
\begin{thaii}[=Theorem \ref{thm1}]   
 Let $M_\sigma$ be a geodesically complete and non-compact spherically symmetric  K\"ahler  manifold of complex dimension $m$. Let 
  $L$ be a positive line bundle over a complex projective manifold $N$ with $\dim_{\mathbb C}N\leq m,$  and  $D\in|L|$  be  of  simple normal crossings.
 Let $f: M_\sigma\rightarrow N$ be a differentiably non-degenerate holomorphic map. Assume that $M_\sigma$ is parabolic. Then for every $\delta>0,$  
  \begin{eqnarray*}
      T_{f}(r,L)+T_{f}(r,K_N)+T(r,\mathscr R_{M_\sigma}) 
     &\leq& \overline{N}_{f}(r,D)+O\big{(}\log^+ T_{f}(r,L)+\delta\log^+r\big{)}
  \end{eqnarray*}
holds for all $r\in(0,\infty)$ outside a set $E_\delta$ of finite Lebesgue measure.
\end{thaii}

We cconsider a defect relation under certain  curvature condition. 
For two  holomorphic line bundles $L_1, L_2$ over $N,$  set  
\begin{eqnarray*}
% \nonumber to remove numbering (before each equatio
\left[\frac{c_1(L_2)}{c_1(L_1)}\right]&=&\inf\left\{t\in\mathbb R: \ \eta_2<t\eta_1;  \ ^\exists\eta_1\in c_1(L_1),\  ^\exists\eta_2\in c_1(L_2) \right\}. 
\end{eqnarray*}
 Let $\Theta_f(D)$  be the \emph{simple defect} of $f$ with respect to $D$ defined  by
$$\Theta_f(D)=1-\limsup_{r\rightarrow R}\frac{\overline{N}_{f}(r,D)}{T_{f}(r,L)}.$$   
\begin{thaiii}[=Corollary \ref{defect}] The  conditions are assumed as same as in Theorem I.   In addition,  assume  that   $M_\sigma$ has non-negative scalar curvature.

$(a)$ For $R=\infty,$ if $T_{f}(r,L)\geq O(\log^+\sigma(r))$ as $r\rightarrow \infty,$ then 
$$\Theta_f(D)\leq \left[\frac{c_1(K^*_N)}{c_1(L)}\right].$$

$(b)$ For $R<\infty,$   if  $\log(R-r)=o(T_{f}(r,L))$ as $r\rightarrow R,$ then
 $$\Theta_f(D)\leq \left[\frac{c_1(K^*_N)}{c_1(L)}\right].$$
\end{thaiii}

\section{Spherically  symmetric manifolds}

\subsection{Laplace operators, polar coordinates and Ricci curvatures}~

\subsubsection{Laplace operators and polar coordinates}~

Let $(M, g)$ be a Riemannian manifold with   Levi-Civita connection  $\nabla.$  The well-known Laplace-Beltrami operator $\Delta_M$ of $\nabla$
is defined by $$\Delta_M=\sum_{i,j}g^{ij}(\nabla_{\partial_i}\nabla_{\partial_j}-\nabla_{\nabla_{\partial_i}\partial_j}),$$
where $\partial_j=\partial/\partial x_j$ and $(g^{ij})$ is the inverse  of $(g_{ij}).$  When  acting on  a function,  $\Delta_M$ has  the  implicit  formula 
$$\Delta_M=\sum_{i,j}\frac{1}{\sqrt{\det(g_{st})}}\frac{\partial}{\partial x_i}\Big(\sqrt{\det(g_{st})}g^{ij}\frac{\partial}{\partial x_j}\Big).$$
 Fix $o\in M,$ one denotes by $B_o(r), S_o(r)$   the geodesic ball and  geodesic sphere of radius $r$ with center at $o$  in $M$ respectively, 
  and  by $r(x)$  the Riemannian distance function of $x$ from $o.$
 Set  $Cut^*(o)=Cut(o)\cup\{o\},$ where $Cut(o)$ is the cut locus of $o.$ 
 For $x\in M\setminus Cut^*(o),$ one can define the polar coordinates
 $(r, \theta)$ of $x$ with respect to the pole $o,$ where $r=r(x)$ is called the polar radius and  $\theta\in \mathbb S^{d-1}$ is called the polar angle which provides the direction $\Gamma_\theta\in T_oM$ of the minimal geodesic connecting $o$ with $x$ at $o,$ 
in which  $d=\dim M,$  $S^{d-1}$ denotes  the  unit sphere in $\mathbb R^d$  centered at the origin.  
Now write the  metric $g$  of $M$ in the polar coordinate form
$$ds^2=dr^2+\sum_{i,j}\tilde g_{ij}d\theta_id\theta_j,$$ 
where $\theta_j$ are  coordinate components  of $\theta,$ and  $\tilde g_{ij}$ 
 is the Riemannian metric  on  $S_o(r)\setminus Cut(o).$ This gives  the Riemannian area element on  $S_o(r)\setminus Cut(o)$ that 
$$dA_r=\sqrt{\det{(\tilde{g}_{st}})}d\theta_1\cdots d\theta_{d-1}.$$
If $\theta_j$ are defined almost everywhere on $\mathbb S^{d-1},$ then we have 
$$Area(S_o(r))=\int_{\mathbb S^{d-1}}\sqrt{\det{(\tilde{g}_{st}})}d\theta_1\cdots d\theta_{d-1}$$
In terms of polar coordinates, the Laplace-Beltrami operator is written as
$$\Delta_M=\frac{\partial^2}{\partial r^2}+\frac{\partial \log\sqrt{\det(\tilde{g}_{st})}}{\partial r}\frac{\partial }{\partial r}+\Delta_{S_o(r)},$$
where $\Delta_{S_o(r)}$ is the induced  Laplace-Beltrami operator on $S_o(r).$ 

Now we turn to Hermitian manifolds. Let $(M, h)$ be a Hermitian manifold with Hermitian connection $\tilde\nabla.$
Note that $M$ can be regarded as a Riemannian manifold with Riemannian metric 
$g=\Re h,$ thus there is also  the Levi-Civita connection $\nabla$  on $M.$ 
Now, extend $\nabla$ linearly to $T_{\mathbb C}M=TM\otimes\mathbb C.$
In general, $\tilde\nabla\not=\nabla$ since  the torsion tensor of $\tilde\nabla$ may not vanish for the general Hermitian manifolds.  Hence, the  Laplace operator  $\tilde\Delta_M$ of $\tilde\nabla$ does not coincide with  the Laplace-Beltrami operator $\Delta_M$ of $\nabla.$ 
However,  the case for $\tilde\nabla=\nabla$ happens when $M$ is a  K\"ahler manifold. 
Consequently, 
$$\Delta_M=\tilde\Delta_M=2\sum_{i,j}h^{i\bar j}\frac{\partial^2}{\partial z_i\partial \bar z_j}$$
 acting on  a function for that $M$ is K\"ahlerian, 
where  $z_j$ are local holomorphic coordinates and  $(h^{i\bar j})$ is the inverse of $(h_{i\bar j}).$ 

\subsubsection{Ricci curvatures}~

Let $(M, h)$ be an $m$-dimensional Hermitian manifold with  K\"ahler form
$$\alpha=\frac{\sqrt{-1}}{\pi}\sum_{i,j}h_{i\bar j}dz_i\wedge d\bar z_j.$$
 The metric $h$ induces a Hermitian metric $\det(h_{i\bar j})$ on the anticanonical bundle $K^*_M.$
 The Chern  form of $K^*_M$ associated to this metric is defined  by
$$\mathscr R_M:=c_1(K^*_M, \det(h_{i\bar j}))=-dd^c\log\det(h_{i\bar j})$$
which is usually called the Ricci form of $M$  due to $\mathscr R_M={\rm{Ric}}(\alpha^m),$
where $$d=\partial+\bar\partial, \ \ \ d^c=\frac{\sqrt{-1}}{4\pi}(\bar\partial-\partial).$$
Assume that $h$ is a K\"ahler metric,  then $\mathscr R_M$ can be written as
$$\mathscr R_M=\frac{\sqrt{-1}}{2\pi}\sum_{i,j}R_{i\bar{j}}dz_i\wedge d\bar{z}_j,$$
where ${\rm{Ric}}_{\mathbb C}=\sum_{i,j}R_{i\bar{j}}dz_i\otimes d\bar{z}_j$ is the complex Ricci curvature tensor of $h.$ 
Regard $M$ as a Riemannian manifold with Riemannina metric $g=\Re h,$ then 
there is also  a real Ricci curvature tensor written as ${\rm{Ric}}_{\mathbb R}=\sum_{i,j}R_{ij}dx_i\otimes dx_j$ of $g.$ 
Denote by $s_{\mathbb C},$ $s_{\mathbb R}$  the scalar curvatures of $h, g$ respectively, i.e., 
$$s_{\mathbb C}=\sum_{i,j}h^{i\bar j}R_{i\bar j}, \ \ \ s_{\mathbb R}=\sum_{i,j}g^{ij}R_{ij}.$$
Then we have
$$s_{\mathbb R}=2s_{\mathbb C}=-\Delta_M\log\det(h_{i\bar j}).$$

\subsection{Spherically  symmetric manifolds}~

Let $(M, g)$ be a Riemannian manifold. We say that $M$ is a manifold with a pole $o$ if $Cut(o)=\emptyset.$  If, in addition, $M$ is complete or geodesically complete, 
then $M$ is diffeomorphic to $\mathbb R^d,$ where $d=\dim M.$
A  Riemannian manifold with a pole $o$ is called a  \emph{spherically symmetric manifold}  if  the induced  metric $\tilde g_{ij}$ on $S_o(r)$ is of the form 
$$\sum_{i,j}\tilde g_{ij}(r,\theta)d\theta_id\theta_j=\sigma^2(r)d\theta^2,$$
where $d\theta^2=d\theta_1^2+\cdots+d\theta_{d-1}^2$ is  the standard Euclidean metric on $S^{d-1}$ and $\sigma$ is a positive smooth function of $r.$
For convenience,  one  uses $M_\sigma$ to denote such  manifolds. 

Let a smooth positive function $\sigma$ on $(0,R)$ with $0<R\leq\infty,$  the necessary and sufficient condition (see \cite{AG}) for that such a manifold exists, is that 
$$\sigma(0)=0, \ \ \sigma'(0)=1.$$
One calls $R$ the  \emph{radius} of $M_\sigma$ with respect to the pole $o.$ Clearly,  $R=\infty$ if $M$ is geodesically  complete and non-compact.
Consider a spherically  symmetric manifold $M_\sigma$ of a pole $o$ and radius $R.$ For $r<R,$ we have 
$$dA_r=\sigma^{d-1}(r)d\theta_1\cdots d\theta_{d-1}.$$
 This means that   
 $$Area(S_o(r))=\int_{\mathbb S^{d-1}}\sigma^{d-1}(r)d\theta_1\cdots d\theta_{d-1}=\omega_{d-1}\sigma^{d-1}(r),$$
 $$Vol(B_o(r))=\omega_{d-1}\int_0^r\sigma^{d-1}(t)dt,$$
 where $\omega_{d-1}$ is the area of $S^{d-1}.$  We also have
 $$\Delta_{M_\sigma}=\frac{\partial^2}{\partial r^2}+(d-1)\frac{\sigma'}{\sigma}\frac{\partial }{\partial r}+\frac{1}{\sigma^2}\Delta_{\theta},$$
 where $\Delta_{\theta}$ is the standard Laplace-Beltrami operator on $S^{d-1}.$
 
\noindent\textbf{Several typical models}

Let $M_\sigma$ be a  spherically  symmetric   manifold of radius $R.$  

$(a)$  If $R=\infty, \sigma(r)=r,$
 then $M_\sigma\cong\mathbb R^d$ (with standard Euclidean metric). The Laplace-Beltrami operator acquires the form
 $$\Delta=\frac{\partial^2}{\partial r^2}+\frac{d-1}{r}\frac{\partial}{\partial r}+\frac{1}{r^2}\Delta_\theta.$$
   
 $(b)$   If $R=\infty, \sigma(r)=\sinh r,$
 then $M_\sigma\cong  H^d,$ where $H^d$  is the $d$-dimensional upper half-space with  hyperbolic metric of sectional curvature $-1$ in
  $\mathbb R^d$. The Laplace-Beltrami operator acquires the form 
 $$\Delta=\frac{\partial^2}{\partial r^2}+(d-1)\cot r\frac{\partial}{\partial r}+\frac{1}{\sin^2r}\Delta_\theta.$$

 $(c)$  If $R=\pi, \sigma(r)=\sin r,$
 then $M_\sigma\cong S^{d}$  (endpoint with $r=\pi$ is added to $M_\sigma$), where $S^d$ is the  $d$-dimensional  unit sphere centered at the origin in $\mathbb R^{d+1}.$ The Laplace-Beltrami operator acquires the form
 $$\Delta=\frac{\partial^2}{\partial r^2}+(d-1)\coth r\frac{\partial}{\partial r}+\frac{1}{\sinh^2r}\Delta_\theta.$$
 
  \subsection{Green functions for spherically symmetric  manifolds}~ 

Let $M_\sigma$ be a $d$-dimensional spherically  symmetric  manifold of a pole $o$ and radius $R.$
Establish a polar coordinate system $(o, r, \theta)$ of $M_\sigma.$
For $0<r<R,$  we shall  compute the harmonic measure $d\pi_o^r(x)$ on $S_o(r)$ with respect to $o,$ as well as the Green function $g_r(o, x)$ of $\Delta_{M_\sigma}/2$ for $B_o(r)$ with pole at $o$ and Dirichlet boundary condition, i.e.,
$$-\frac{1}{2}\Delta_{M_\sigma}g_r(o,x)=\delta_o(x) \ \text{for} \ x\in B_o(r); \ \ g_r(o,x)=0 \ \text{for} \ x\in S_o(r),$$
where $\delta_o$ is the Dirac function. 
\begin{lemma}\label{asdf}  For $0<r<R,$ we have
$$d\pi_o^r(x)=\frac{d\theta_1\cdots d\theta_{d-1}}{\omega_{d-1}}, \ \  g_r(o, x)=\frac{2}{\omega_{d-1}}\int_{r(x)}^r\frac{dt}{\sigma^{d-1}(t)},$$
where $\omega_{d-1}$ is the area of the unit sphere in $\mathbb R^d$ with $d\geq2.$ 
\end{lemma}
\begin{proof}  By the  property of spherically  symmetric   manifolds,  the induced  area measure   $$dS_r(x)=\sigma^{d-1}(r)d\theta_1\cdots d\theta_{d-1}$$ 
 on $S_o(r)$ is a rotationally invariant one with respect to $o.$ On the other hand, $dS_r(x)/Area(S_o(r))$ is  a probability measure on $S_o(r).$ Thus,
$$d\pi_o^r(x)=\frac{dS_r(x)}{Area(S_o(r))}=\frac{d\theta_1\cdots d\theta_{d-1}}{\omega_{d-1}}.$$
Notice the relation 
$$-\frac{1}{2}\frac{\partial g_r(o,x)}{\partial n}=\frac{d\pi^r_o(x)}{dS_r(x)},$$
where $\partial/\partial n$ is the inward normal derivative on $S_o(r).$ Then 
$$\frac{\partial g_r(o,x)}{\partial n}=-\frac{2}{\omega_{d-1}\sigma^{d-1}(r)}.$$
On the other hand,  $$-\frac{1}{2}\Delta_{M_\sigma}g_r(o,x)=\delta_o(x)$$
 for  $x\in B_o(r).$ Combine the above two equations, it is trivial to confirm that
 $$g_r(o, x)=\frac{2}{\omega_{d-1}}\int_{r(x)}^r\frac{dt}{\sigma^{d-1}(t)}.$$
\end{proof}

In what follows, we  give two  examples to compute  Green functions.

\noindent\textbf{Example 1.} 
   $M_\sigma=\mathbb R^d$  (with standard Euclidean metric)

Take $o$ as the coordinate origin of $\mathbb R^d$. By  $\sigma(r)=r$ and $r(x)=\|x\|,$ one has    
$$g_r(o, x)=\frac{2}{\omega_{d-1}}\int_{\|x\|}^r\frac{dt}{t^{d-1}},$$
which can be computed easily. 

\noindent\textbf{Example 2.} 
  $M_\sigma=\mathbb H$ (Poincar\'e upper half-plane, i.e., $H^2$ with Poincar\'e metric) 

Take $o=(0, \sqrt{-1}).$ Let  $\phi: \mathbb D\rightarrow \mathbb H$ be the biholomorphic map as follows   
$$\phi(z)=\frac{1-\sqrt{-1}z}{z-\sqrt{-1}}.$$
Note that     
 $\phi^*h$ is the Poincar\'e metric on $\mathbb D,$ where $h$ is the   Poincar\'e metric on $\mathbb H.$  By $\sigma(r)=\sinh r$ and $\omega_1=2\pi,$ we see that 
 \begin{equation}\label{qqq3}
 g_r(o, x)=\frac{2}{\pi}\int_{r(x)}^r\frac{dt}{e^t-e^{-t}}=\frac{1}{\pi}\log\frac{(e^r-1)(e^{r(x)}+1)}{(e^r+1)(e^{r(x)}-1)},
 \end{equation}
 where
 $$r(x)=\log\frac{1+|\phi^{-1}(x)|}{1-|\phi^{-1}(x)|}.$$
 
\section{Holomorphic maps on spherically  symmetric  K\"ahler  manifolds}~
 \subsection{Nevanlinna's functions and First Main Theorem}~
 \subsubsection{Nevanlinna's functions}~
  
  We  extend the notion of Nevanlinna's functions containing  characteristic function, counting function and proximity function to  spherically  symmetric K\"ahler manifolds. 
 Let $(M_\sigma, h)$ be a spherically  symmetric  K\"ahler  manifold of complex dimension $m,$ with a pole $o$ and radius $R.$  Then the  K\"ahler form of $M$ is written as
 $$\alpha=\frac{\sqrt{-1}}{\pi}\sum_{i.j}h_{i\bar j}dz_i\wedge d\bar z_{j}.$$
Fix a $r_0$ such that $0<r_0<R.$   For any (1,1)-form $\eta$ on $M_\sigma,$ define formally the notation
 $$T(r,\eta)=\int_{r_0}^{r}\frac{dt}{\sigma^{2m-1}(t)}\int_{B_o(t)}\eta\wedge\alpha^{m-1}.$$
 Let $f:M_\sigma\rightarrow N$ be a holomorphic map into a complex projective manifold $N,$ and  $(L, h_L)$  be  a positive Hermitian line bundle  over $N.$ 
Let $|L|$ be the
\emph{complete linear system} of  all effective divisors $D_s$ with $s\in  H^0(M,L),$ where $D_s$ denotes  the zero divisor of a section $s.$  
   For $r_0<r<R,$ the \emph{characteristic function} of 
 $f$ with respect to $L$ is defined by 
 $$T_{f}(r,L)=T_f(r,c_1(L, h_L))$$
 up to a bounded term.   
 Since $M_\sigma$ is K\"ahlerian, then
$$\Delta_M=2\sum_{i,j}h^{i\bar j}\frac{\partial^2}{\partial z_i\partial \bar z_j}.$$
Thus, 
 $$\Delta_{M_\sigma}\log h_L=-4m\frac{f^*c_1(L,h)\wedge\alpha^{m-1}}{\alpha^m}.$$
In terms of Green function, we have 
 \begin{eqnarray*}
T_{f}(r,L)&=& -\frac{1}{4}\int_{B_o(r)}g_{r}(o,x)\Delta_{M_\sigma}\log h_L(x)dV(x) \\
&&+\frac{1}{4}\int_{B_o(r_0)}g_{r_0}(o,x)\Delta_{M_\sigma}\log h_L(x)dV(x),
\end{eqnarray*}
 where $dV=\pi^m\alpha^m/m!$ is the Riemannian volume element of $M_\sigma.$
 Now for a divisor  $D\in|L|,$  the \emph{counting function} of $f$ with respect to $D$ is defined by
 $$N_{f}(r,D)=\int_{r_0}^r\frac{dt}{\sigma^{2m-1}(t)}\int_{f^*D\cap B_o(t)}\alpha^{m-2}.$$
 Let $s_D$ be the canonical section of $L,$ it is of zero divisor $D.$    
 Write $s_D=\tilde s_De,$ where $e$ is a local holomorphic frame of $L.$ By Poincar\'e-Lelong formula \cite{gri},  $N_{f,D}(r)$ has an alternate expression 
  \begin{eqnarray*}
N_{f}(r,D)&=&\frac{1}{4}\int_{B_o(r)}g_r(o,x)\Delta_{M_\sigma}\log |\tilde s_D\circ f(x)|^2dV(x) \\
&&-\frac{1}{4}\int_{B_o(r_0)}g_{r_0}(o,x)\Delta_{M_\sigma}\log |\tilde s_D\circ f(x)|^2dV(x) 
\end{eqnarray*}
 In a similar way, we  define the \emph{simple counting function} $\overline N_{f}(r,D)$ for Supp$f^*D.$

For the \emph{proximity function} of $f$ with respect to $D,$  we use the definition  
\begin{eqnarray*}
m_{f}(r,D)&=&\int_{S_o(r)}\log\frac{1}{\|s_D\circ f(x)\|}d\pi^r_o(x) \\
&=& \frac{1}{\omega_{2m-1}}\int_{S^{2m-1}}\log\frac{1}{\|s_D\circ f(r, \theta)\|}d\theta_1\cdots d\theta_{2m-1},
\end{eqnarray*}
where $(r, \theta)$ stands for  the polar coordinate of $x$ with respect to the pole $o,$   $S^{2m-1}$ is the unit sphere in $\mathbb R^{2m},$ and 
$\omega_{2m-1}$ is the area of $S^{2m-1}.$
 The  last equality is due to Lemma \ref{asdf}.  Denote by $\mathscr R_{M_\sigma}=-dd^c\log\det(h_{i\bar j})$  the Ricci form of  $M_\sigma,$ then 
 \begin{eqnarray}
 T(r,\mathscr R_{M_\sigma})&=&\int_{r_0}^{r}\frac{dt}{\sigma^{2m-1}(t)}\int_{B_o(t)}\mathscr R_{M_\sigma}\wedge\alpha^{m-1}  \nonumber \\
&=& \frac{1}{2}\int_{B_o(r)}g_r(o,x)s_{M_\sigma}(x)dV(x)-\frac{1}{2}\int_{B_o(r_0)}g_{r_0}(o,x)s_{M_\sigma}(x)dV(x), \label{scalar}
\end{eqnarray}
 where $s_{M_\sigma}$ is the scalar curvature of $M_\sigma,$ see Section 2.1.2.

\noindent {\textbf{Remark.}} When $M_\sigma=\mathbb C^m$ with standard Euclidean metric, we have $R=\infty$ and $\sigma(r)=r.$ 
Since $d\pi^r_o(z)=d^c\log\|z\|^2\wedge \left(dd^c\log\|z\|^2\right)^{m-1},$  the generalized  definition of Nevanlinna's functions agrees with the classical one, see \cite{noguchi, ru}.

 \subsubsection{First Main Theorem}~

In the classical Nevanlinna theory,   Green-Jensen formula \cite{noguchi, ru} deduces  the Nevanlinna's First Main Theorem. 
For a  meromorphic function  defined on  a  complex manifold,  we need  Dynkin formula (see \cite{at, NN, itoo})   which plays the similar  role as Green-Jensen formula. 
 Let's  introduce  a simple version of Dynkin formula which is viewed  as a special case of  the original probabilistic version  via Brownian motion, see, e.g., \cite{at, atsuji, carne, dong, NN, itoo}.

\noindent\textbf{Dynkin formula.} \emph{Let $u$ be a  function of  $\mathscr C^2$-class except at most a polar set of singularities on a Riemannian manifold $M.$ For $0<r_0<r$ or  $0\leq r_0<r$ with $u(o)\not=\infty,$ 
we have 
\begin{eqnarray*}
&&\int_{S_o(r)}u(x)d\pi_o^r(x)-\int_{S_o(r_0)}u(x)d\pi_o^r(x) \\
&=& \frac{1}{2}\int_{B_o(r)}g_r(o, x)\Delta_Mu(x)dV(x)- \frac{1}{2}\int_{B_o(r_0)}g_{r_0}(o, x)\Delta_Mu(x)dV(x)
\end{eqnarray*}
where $B_o(r), S_o(r)$ are   geodesic ball and geodesic sphere of radius $r$ centered at $o$ respectively,  $g_r(o, x)$ is the Green function of $\Delta_M/2$ for 
$B_o(r)$ with pole at $o$ and Dirichlet boundary condition, and $d\pi_o^r$ is the harmonic metric on $S_o(r)$ with respect to $o.$ Here, $\Delta_M u$ should be understood as distributions.}

Particularly, when $M=\mathbb C^m,$  Dynkin formula coincides with Green-Jensen formula (see \cite{noguchi, ru}). 
According to the definition of Nevanlinna's functions and  Dynkin formula, 
we  can easily obtain   the First Main Theorem   as follows 
$$ \text{F. M. T.}  \ \  \ T_{f}(r,L)=m_{f}(r,D)+N_{f}(r,D)+O(1).$$

\subsection{Logarithmic Derivative Lemma}~ 

Let $(M_\sigma, h)$ be a Hermitian model  manifold of complex dimension $m,$ with a pole $o$ and radius $R.$ 
\begin{lemma}[\cite{ru-sibony}]\label{borel}  Let $\gamma$ be an integrable function on $(0, R)$ with $\int_0^R\gamma(r)dr=\infty.$  Let $h$ be a nondecreasing function of $\mathscr C^1$-class
on $(0, R).$ Assume that $\lim_{r\rightarrow R}h(r)=\infty$ and $h(r_0)>0$ for some $r_0\in (0, R).$ Then for every $\delta>0$  
$$h'(r)\leq h^{1+\delta}(r)\gamma(r)$$
holds for all $r\in(0,R)$ outside a set $E_\delta$ with $\int_{E_{\delta}}\gamma(r)dr<\infty.$ In particular, when $R=\infty,$ we can take $\gamma=1.$ Then  for every $\delta>0$ 
$$h'(r)\leq h^{1+\delta}(r)$$
holds for all $r\in(0,\infty)$ outside a set $E_\delta$ of finite Lebesgue measure.
\end{lemma}

Let $\Gamma$ be a locally integrable function on $M_\sigma.$ Set
  $$E_\Gamma(r)=\int_{S_o(r)}\Gamma(x) d\pi_o^r(x)  , \ \ \ 
 T_{\Gamma}(r)=\int_{r_0}^{r}\frac{dt}{\sigma^{2m-1}(t)}\int_{B_o(t)}\Gamma(x)\alpha^{m}.$$

 We need the following so-called Calculus Lemma
\begin{lemma}\label{CL}  Let  $\gamma$ be an integrable function on $(0, R)$ with $\int_0^R\gamma(r)dr=\infty.$ Let $\Gamma$ be a locally integrable function on $M_\sigma.$ Then for every $\delta>0$
$$E_{\Gamma}(r)\leq \frac{\pi ^m}{\omega_{2m-1}m!}\sigma^{(2m-1)\delta}(r)\gamma^{2+\delta}(r)T^{(1+\delta)^2}_{\Gamma}(r)$$
holds for all $r\in(0,R)$ outside a set $E_\delta$ with $\int_{E_{\delta}}\gamma(r)dr<\infty,$ where $\omega_{2m-1}$ is the area of the unit sphere $S^{2m-1}$ in $\mathbb R^{2m}.$
\end{lemma}
\begin{proof} Notice that 
$$\int_{B_o(r)}\Gamma(x)\alpha^{m}=\frac{m!\omega_{2m-1}\sigma^{2m-1}(r)}{\pi^m}\int_0^rdt\int_{S_o(t)}\Gamma(x)d\pi_o^r(x),$$
 then we have  
$$\frac{d}{dr}\Big(\sigma^{2m-1}\frac{dT_\Gamma}{dr}\Big)=\frac{m!\omega_{2m-1}\sigma^{2m-1}}{\pi^m}E_\Gamma.$$
Using  Lemma \ref{borel} twice (first to $\sigma^{2m-1}T'_\Gamma$ and then to $T_\Gamma$), then we  can prove the lemma.
\end{proof}

Let $\psi$ be a meromorphic function on  $M_\sigma.$ 
The norm of the gradient of $\psi$ is defined by
$$\|\nabla_{M_\sigma}\psi\|^2=2\sum_{i,j}h^{i\overline j}\frac{\partial\psi}{\partial z_i}\overline{\frac{\partial \psi}{\partial  z_j}}.$$
 Identify $\psi$  with a meromorphic mapping into $\mathbb P^1(\mathbb C).$  
  The  characteristic function of $\psi$ with respect to the Fubini-Study form $\omega_{FS}$ on  $\mathbb P^1(\mathbb C)$  is defined by
 $$T_{\psi}(r)=\int_{r_0}^{r}\frac{dt}{\sigma^{2m-1}(t)}\int_{B_o(t)}f^*\omega_{FS}\wedge\alpha^{m-1}.$$
Let
  $i:\mathbb C\hookrightarrow\mathbb P^1(\mathbb C)$ be an inclusion, then it induces  a (1,1)-form $i^*\omega_{FS}$ on $\mathbb C.$
 The \emph{Ahlfords characteristic function} of $\psi$  is defined by
 $$\hat T_{\psi}(r)=\int_{r_0}^{r}\frac{dt}{\sigma^{2m-1}(t)}\int_{B_o(t)}f^*(i^*\omega_{FS})\wedge\alpha^{m-1}.$$
Moreover, we  define the \emph{Nevanlinna characteristic function}
$$T(r,\psi)=m(r,\psi)+N(r,\psi),$$
where
$$m(r,\psi)=\int_{S_o(r)}\log^+|\psi(x)|d\pi^r_o(x), \ \ 
N(r,\psi)= \int_{r_0}^r\frac{dt}{\sigma^{2m-1}(t)}\int_{f^*\infty\cap B_o(t)}\alpha^{m-2}.$$
 It is trivial to confirm that  $\hat{T}_\psi(r)\leq T_\psi(r)$ and $T(r,\psi)=\hat{T}_\psi(r)+O(1).$  Thus,  we obtain 
$$T(r,\psi)\leq T_\psi(r)+O(1).$$

  On $\mathbb P^1(\mathbb C),$  take a singular metric
$$\Phi=\frac{1}{|\zeta|^2(1+\log^2|\zeta|)}\frac{\sqrt{-1}}{4\pi^2}d\zeta\wedge d\bar \zeta.$$
A direct computation gives that
\begin{equation}\label{ada}
\int_{\mathbb P^1(\mathbb C)}\Phi=1, \ \ \ 4m\pi\frac{\psi^*\Phi\wedge\alpha^{m-1}}{\alpha^m}=\frac{\|\nabla_{M_\sigma}\psi\|^2}{|\psi|^2(1+\log^2|\psi|)}.
\end{equation}
Set
 $$T_{\psi}(r,\Phi)=\int_{r_0}^{r}\frac{dt}{\sigma^{2m-1}(t)}\int_{B_o(t)}\psi^*\Phi\wedge\alpha^{m-1}.$$
\begin{lemma}\label{oo12} We have
$$T_{\psi}(r,\Phi)\leq T(r,\psi)+O(1).$$
\end{lemma}
\begin{proof} It yields from Fubini theorem that 
\begin{eqnarray*}
T_{\psi}(r,\Phi)&=&\int_{\mathbb P^1(\mathbb C)}\Phi(\zeta)\int_{r_0}^{r}\frac{dt}{\sigma^{2m-1}(t)}\int_{\psi^*\zeta\cap B_o(t)}\alpha^{m-2} \\
&=&\int_{\mathbb P^1(\mathbb C)}N_\psi(r,\zeta)\Phi(\zeta)  \\
&\leq&\int_{\mathbb P^1(\mathbb C)}\big{(}T(r,\psi)+O(1)\big{)}\Phi \\
&=& T(r,\psi)+O(1).
\end{eqnarray*}
\end{proof}
\begin{lemma}\label{999a1}  Let $\gamma$ be an integrable  function on $(0, R)$ with $\int_0^R\gamma(r)dr=\infty.$   
Let  $\psi\not\equiv0$ be a meromorphic function on  $M_\sigma.$ 
Then for every $\delta>0$ 
 \begin{eqnarray*}
 && \int_{S_o(r)}\log^+\frac{\|\nabla_{M_\sigma}\psi(x)\|^2}{|\psi(x)|^2(1+\log^2|\psi(x)|)}d\pi_o^r(x) \\
  &\leq& (1+\delta)^2\log^+ T(r,\psi)+(2+\delta)\log^+\gamma(r)+(2m-1)\delta\log^+\sigma(r)+O(1)
\end{eqnarray*}
holds for all $r\in(0,R)$ outside a set $E_\delta$ with $\int_{E_{\delta}}\gamma(r)dr<\infty.$
\end{lemma}
\begin{proof} By Jensen inequality
\begin{eqnarray*}
  && \int_{S_o(r)}\log^+\frac{\|\nabla_{M_\sigma}\psi(x)\|^2}{|\psi(x)|^2(1+\log^2|\psi(x)|)}d\pi_o^r(x) \\
   &\leq&  \int_{S_o(r)}\log\Big{(}1+\frac{\|\nabla_{M_\sigma}\psi(x)\|^2}{|\psi(x)|^2(1+\log^2|\psi(x)|)}\Big{)}d\pi_o^r(x) \nonumber \\
    &\leq& \log^+\int_{S_o(r)}\frac{\|\nabla_{M_\sigma}\psi(x)\|^2}{|\psi(x)|^2(1+\log^2|\psi(x)|)}d\pi_o^r(x)+O(1). \nonumber
\end{eqnarray*}
Applying Dykin formula, Lemma  \ref{CL} and (\ref{ada}) to get 
\begin{eqnarray*}
   &&\log^+\int_{S_o(r)}\frac{\|\nabla_{M_\sigma}\psi(x)\|^2}{|\psi(x)|^2(1+\log^2|\psi(x)|)}d\pi_o^r(x)  \\
   &\leq& (1+\delta)^2\log^+\int_{r_0}^{r}\frac{dt}{\sigma^{2m-1}(t)}\int_{B_o(t)}\frac{\|\nabla_{M_\sigma}\psi(x)\|^2}{|\psi(x)|^2(1+\log^2|\psi(x)|)}\alpha^{m} \\ 
   && +(2+\delta)\log^+\gamma(r)+(2m-1)\delta\log^+\sigma(r)+O(1)
   \nonumber \\
   &=& (1+\delta)^2\log^+ T(r,\psi)+(2+\delta)\log^+\gamma(r)+(2m-1)\delta\log^+\sigma(r)+O(1). \nonumber
\end{eqnarray*}
Combining the above, the lemma is proved.
\end{proof}

Define
$$m\left(r,\frac{\|\nabla_{M_\sigma}\psi\|}{|\psi|}\right)=\int_{S_o(r)}\log^+\frac{\|\nabla_{M_\sigma}\psi(x)\|}{|\psi|(x)}d\pi^r_o(x).$$

We have the following Logarithmic Derivative Lemma
\begin{theorem}\label{LDL}   Let $\gamma$ be an integrable  function on $(0, R)$ with $\int_0^R\gamma(r)dr=\infty.$  
 Let  $\psi\not\equiv0$ be a meromorphic function on  $M_\sigma.$ Then for every $\delta>0$ 
\begin{eqnarray*}
    m\left(r,\frac{\|\nabla_{M_\sigma}\psi\|}{|\psi|}\right) 
   &\leq&   
    \frac{2+(1+\delta)^2}{2}\log^+ T(r,\psi)+\frac{(2+\delta)}{2}\log^+\gamma(r) \\  
    && +\frac{(2m-1)\delta}{2}\log^+\sigma(r)+O(1)
\end{eqnarray*}
holds for all $r\in(0,R)$ outside a set $E_\delta$ with $\int_{E_{\delta}}\gamma(r)dr<\infty.$
\end{theorem}
\begin{proof}  Notice that 
\begin{eqnarray*}
    m\left(r,\frac{\|\nabla_{M_\sigma}\psi\|}{|\psi|}\right)
   &\leq& \frac{1}{2}\int_{S_o(r)}\log^+\frac{\|\nabla_{M_\sigma}\psi(x)\|^2}{|\psi(x)|^2(1+\log^2|\psi(x)|)}d\pi^r_o(x) \\
  && +\frac{1}{2}\int_{S_o(r)}\log^+\big{(}1+\log^2|\psi(x)|\big{)}d\pi^r_o(x) \\
     &\leq&   \frac{1}{2}\int_{S_o(r)}\log^+\frac{\|\nabla_{M_\sigma}\psi(x)\|^2}{|\psi(x)|^2(1+\log^2|\psi(x)|)}d\pi^r_o(x)  \\
  && +\int_{S_o(r)}\log\Big{(}1+\log^+|\psi(x)|+\log^+\frac{1}{|\psi(x)|}\Big{)}d\pi^r_o(x).
\end{eqnarray*}
Lemma \ref{999a1} implies that for every $\delta>0$ 
 \begin{eqnarray*}
 && \frac{1}{2}\int_{S_o(r)}\log^+\frac{\|\nabla_{M_\sigma}\psi(x)\|^2}{|\psi(x)|^2(1+\log^2|\psi(x)|)}d\pi_o^r(x) \\
  &\leq& \frac{(1+\delta)^2}{2}\log^+ T(r,\psi)+\frac{(2+\delta)}{2}\log^+\gamma(r)+\frac{(2m-1)\delta}{2}\log^+\sigma(r)+O(1)
\end{eqnarray*}
holds for all $r\in(0,R)$ outside a set $E_\delta$ with $\int_{E_{\delta}}\gamma(r)dr<\infty.$
Using Jensen inequality, it follows that 
\begin{eqnarray*}
&& \int_{S_o(r)}\log\Big{(}1+\log^+|\psi(x)|+\log^+\frac{1}{|\psi(x)|}\Big{)}d\pi^r_o(x) \\
 &\leq& \log\int_{S_o(r)}\Big{(}1+\log^+|\psi(x)|+\log^+\frac{1}{|\psi(x)|}\Big{)}d\pi^r_o(x) \\
   &\leq& \log \Big{(}m(r,\psi) +m\big(r,\frac{1}{\psi}\big)\Big{)}+O(1)  \\
   &\leq& \log^+ T(r,\psi)+O(1).
\end{eqnarray*}
 Combining the above, we prove  the theorem.
\end{proof}

\subsection{Second Main Theorem}~

This subsection aims to prove the following  Second Main  Theorem
\begin{theorem}\label{main theorem}  Let  $M_\sigma$ be a spherically  symmetric K\"ahler manifold of complex dimension $m,$ with a pole $o$ and radius $R.$    Let $L$ be a positive line bundle over a complex projective manifold $N$ with $\dim_{\mathbb C}N\leq m,$  and  $D\in|L|$  be  of  simple normal crossings.
 Let $f:M_\sigma\rightarrow N$ be a differentiably non-degenerate holomorphic map. Then for every $\delta>0$  
  \begin{eqnarray*}
     && T_{f}(r,L)+T_{f}(r,K_N)+T(r,\mathscr R_{M_\sigma}) \\
     &\leq& \overline{N}_{f}(r,D)+O\big{(}\log^+ T_{f}(r,L)+\log^+\gamma(r)+\delta\log^+\sigma(r)\big{)}
  \end{eqnarray*}
holds for all $r\in(0,R)$ outside a set $E_\delta$ with $\int_{E_{\delta}}\gamma(r)dr<\infty,$ where 
$\gamma$ is an integrable  function on $(0, R)$ such that $\int_0^R\gamma(r)dr=\infty.$  We have two cases$:$

$(a)$ For $R=\infty,$   we  take $\gamma(r)=1.$  Then for every $\delta>0$  
  \begin{eqnarray*}
     && T_{f}(r,L)+T_{f}(r,K_N)+T(r,\mathscr R_{M_\sigma}) \\
     &\leq& \overline{N}_{f}(r,D)+O\big{(}\log^+ T_{f}(r,L)+\delta\log^+\sigma(r)\big{)}
  \end{eqnarray*}
holds for all $r\in(0,\infty)$ outside a set $E_\delta$ of finite Lebesgue measure. 

$(b)$ For $R<\infty,$   we  take $\gamma(r)=\frac{1}{R-r}.$  Then  for every $\delta>0$  
  \begin{eqnarray*}
    && T_{f}(r,L)+T_{f}(r,K_N)+T(r,\mathscr R_{M_\sigma}) \\
     &\leq& \overline{N}_{f}(r,D)+O\Big{(}\log^+ T_{f}(r,L)+\log\frac{1}{R-r}\Big{)}
  \end{eqnarray*}
 holds for all $r\in(0,R)$ outside a set $E_\delta$ with $\int_{E_{\delta}}(R-r)^{-1}dr<\infty.$
\end{theorem}

We first give several  consequences  before proving Theorem \ref{main theorem}.

\noindent{\textbf{1. Three classical  consequences}}

$(a)$ $M_\sigma=\mathbb C^m$ (with standard Euclidean metric)

The case implies that  $R=\infty$ and $\sigma(r)=r.$  Then 
 $$T_{f}(r,L)=\int_{r_0}^{r}\frac{dt}{r^{2m-1}(t)}\int_{B_o(t)}f^*c_1(L, h_L)\wedge\alpha^{m-1},$$
  which coincides with the classical  characteristic function.
Since $\mathbb C^m$ has sectional curvature 0,  then  conclusion $(a)$ in Theorem \ref{main theorem} derives immediately the classical result of 
Carlson-Griffiths-King (see \cite{gri, gri1}) as follows
\begin{cor}[Carlson-Griffiths-King]\label{cgk}   Let $L$ be a positive line bundle over a complex projective manifold $N$ with $\dim_{\mathbb C}N\leq m,$  and  $D\in|L|$  be  of  simple normal crossings.
 Let $f: \mathbb C^m\rightarrow N$ be a differentiably non-degenerate holomorphic map. Then for every $\delta>0$  
  \begin{eqnarray*}
      T_{f}(r,L)+T_{f}(r,K_N)
     &\leq& \overline{N}_{f}(r,D)+O\big{(}\log^+ T_{f}(r,L)+\delta\log^+r\big{)}
  \end{eqnarray*}
holds for all $r\in(0,\infty)$ outside a set $E_\delta$ of finite Lebesgue measure. 
\end{cor}
More generalizations of Corollary \ref{cgk}
were done by Sakai \cite{Sakai} in terms of Kodaira dimension
 and  by Shiffman \cite{Shiff} in the case of the singular divisor, see also  \cite{dong, green, hu, Lang, Nchi, Noguchi, ru1, stoll, wong}.

$(b)$ $M_\sigma= \mathbb B^m$ (unit  ball of  complex dimension $m$ with standard Euclidean  metric)

The case implies that  $R=1, \sigma(r)=r$ and $\mathscr R_{\mathbb B^m}=0.$ We also have
 $$T_{f}(r,L)=\int_{r_0}^{r}\frac{dt}{r^{2m-1}(t)}\int_{B_o(t)}f^*c_1(L, h_L)\wedge\alpha^{m-1},$$
 which agrees with the classical  characteristic function.
By Theorem \ref{main theorem} $(b)$,  it  yields that
\begin{cor}\label{ccc}   Let $L$ be a positive line bundle over a complex projective manifold $N$ with $\dim_{\mathbb C}N\leq m,$  and  $D\in|L|$  be  of  simple normal crossings.
 Let $f: \mathbb B^m\rightarrow N$ be a differentiably non-degenerate holomorphic map, where $\mathbb B^m$ is the unit ball with standard Euclidean metric. Then for every $\delta>0$  
  \begin{eqnarray*}
      T_{f}(r,L)+T_{f}(r,K_N)
     &\leq& \overline{N}_{f}(r,D)+O\Big{(}\log^+ T_{f}(r,L)+\log\frac{1}{1-r}\Big{)}
  \end{eqnarray*}
holds for all $r\in(0,1)$ outside a set $E_\delta$ with $\int_{E_{\delta}}(1-r)^{-1}dr<\infty.$
\end{cor}

$(c)$ $M_\sigma= \mathbb H$ or $\mathbb D$ (Poincar\'e upper half-plane or Poincar\'e disc)

$\mathbb H$ and $\mathbb D$  are  two  representative  models in hyperbolic geometry,  marking many essential  differences from Euclidean geometry.  It is important to study the value distribution of a meromorphic function on Poincar\'e models, which  provides an effective  tool to investigate 
the modular functions
\begin{equation*}
    g(\tau)=\sum_{n\geq m}c_ne^{2\pi i n\tau}
\end{equation*}
on $ \mathbb H$, where $g$ is called a modular (resp. cusp)  form if $m=0$ (resp. $m=1$ ), see, e.g., \cite{hu, hu1}.

In what follows, one  establishes a  Second Main Theorem of meromorphic functions on Poincar\'e models in the sense of hyperbolic metric.
When $M_\sigma= \mathbb H$ or $\mathbb D$ with Poincar\'e metric, we have  $R=\infty$ and $\sigma(r)=\sinh r.$ Let 
 $f$ be a meromorphic function on  $\mathbb H$ or $\mathbb D.$ In this situation, we define the  Ahlfords characteristic function
  \begin{eqnarray*}
  T_{f}(r)&=&\int_{r_0}^{r}\frac{dt}{\sinh t}\int_{B_o(t)}dd^c\log(1+|f(x)|^2) \\
  &=& \frac{1}{4}\int_{B_o(r)}g_r(o, x)\Delta\log(1+|f(x)|^2dV(x) \\
  &&-\frac{1}{4}\int_{B_o(r_0)}g_{r_0}(o, x)\Delta\log(1+|f(x)|^2dV(x)
  \end{eqnarray*}
  where $\Delta$ is the Laplace-Beltrami operator on  $\mathbb H$ or $\mathbb D.$
  Adopt the spherical distance  $\|\cdot , \cdot\|$ on $\mathbb P^1(\mathbb C).$ By definition,  the proximity function is that 
    \begin{eqnarray*}
m_{f}(r, a)&=& \int_{S_o(r)}\log \frac{1}{\|f(x), a\|}d\pi_o^{r}(x)   \\ 
&=& \frac{1}{2\pi}\int_{0}^{2\pi}\log \frac{1}{\|f(r, \theta), a\|}d\theta
  \end{eqnarray*}
for any point $a\in \overline\CC,$ where $(r, \theta)$ is the polar coordinate of $x$ with respect to the pole  $o.$  According to the definition (see Section 2.1.1), the counting function is defined by 
$$N_{f}(r, a) = \int_{r_0}^r\frac{n_f(t, a)}{\sinh t}dt$$ 
for a point $a\in \overline\CC,$ where $n_f(r, a)$ denotes the number of the zeros of $f-a$ in $B_o(r)$ counting multiplicities.
By Dynkin formula, we have the First Main Theorem  
 $$T_{f}(r)=m_{f}(r, a)+N_{f}(r, a)+O(1).$$
 For the sake of intuition of $T_f(r)$, we introduce the Nevanlinna characteristic function
 $$T(r, f):=m(r, f)+N(r, f),$$
 where
 $$m(r, f)=\frac{1}{2\pi}\int_{0}^{2\pi}\log^+|f(r, \theta)|d\theta, \ \  N(r, f)= \int_{r_0}^r\frac{n_f(t, \infty)}{\sinh t}dt.$$
Note that  $$m_f(r, \infty)=m(r, f)+O(1), \ \ N_f(r, \infty)=N(r, f),$$ hence 
 $$T_f(r)=T(r, f)+O(1).$$
 Namely, 
 $$T_f(r)=\frac{1}{2\pi}\int_{0}^{2\pi}\log^+|f(r, \theta)|d\theta+ \int_{r_0}^r\frac{n_f(t, \infty)}{\sinh t}dt+O(1).$$
 
  Let $\omega_{FS}=dd^c\log\|w\|^2$ be the Fubini-Study form on $\mathbb P^1(\mathbb C),$ where $w=[w_0: w_1]$ is the homogeneous  coordinate of   $\mathbb P^1(\mathbb C).$
 Let $a_1, \cdots, a_q$ be different points in $\mathbb P^1(\mathbb C),$ and  
 regard $f=f_1/f_0=[f_0:f_1]$ as a holomorphic map into $\mathbb P^1(\mathbb C).$   
 Theorem \ref{main theorem} $(a)$ gives  that 
$$ (q-2)T_{f, \omega_{FS}}(r)+T_{\mathscr R}(r) 
     \leq \sum_{j=1}^q\overline{N}_{f,a_j}(r)+O\big{(}\log^+ T_{f, \omega_{FS}}(r)+\delta\log^+\sinh r\big{)},
$$
which implies that 
$$ (q-2)T_{f}(r)+T_{\mathscr R}(r) 
     \leq \sum_{j=1}^q\overline{N}_{f}(r, a_j)+O\big{(}\log^+ T_{f}(r)+\delta\log^+\sinh r\big{)}. 
$$
 As noted in Sections 2.2 and 2.3,  
  $$Area(S_o(r))=2\pi\sinh r, \ \ \  g_r(o, x)=\frac{1}{\pi}\log\frac{(e^r-1)(e^{r(x)}+1)}{(e^r+1)(e^{r(x)}-1)}.$$
Since $s_{\mathbb H}=s_{\mathbb D}=-1,$ then a direct computation leads to 
  \begin{eqnarray*}
T(r,\mathscr R) &=& \frac{1}{2}\int_{B_o(r)}g_r(o,x)s(x)dV(x)-\frac{1}{2}\int_{B_o(r_0)}g_{r_0}(o,x)s(x)dV(x)  \\
&=&   -\frac{1}{2}\int_{0}^rdt\int_{S_o(t)}g_r(o,x)dA_t(x) +\frac{1}{2}\int_{0}^{r_0}dt\int_{S_o(t)}g_{r_0}(o,x)dA_t(x)   \\
&=& -\int_{0}^r\log\frac{(e^r-1)(e^{t}+1)}{(e^r+1)(e^{t}-1)}\cdot \sinh t dt \\
&&+\int_{0}^{r_0}\log\frac{(e^{r_0}-1)(e^{t}+1)}{(e^{r_0}+1)(e^{t}-1)}\cdot \sinh t dt \\
&\geq& -r+O(1),
  \end{eqnarray*}
 which implies that 
 $$-T(r,\mathscr R)\leq r+O(1).$$
  Moreover, 
 $$\log^+\sinh r\leq r+O(1).$$
Therefore, we conclude that  
\begin{cor}\label{ccc1}   
 Let $f$ be a  nonconstant meromorphic function on $\mathbb H$  or  $\mathbb D$, where $\mathbb H, \mathbb D$ are the Poincar\'e upper half-plane and Poincar\'e disc respectively. 
 Let $a_1, \cdots, a_q$ be distinct points in $\overline\CC$
 Then for every $\delta>0$  
  \begin{eqnarray*}
      (q-2)T_{f}(r)
     &\leq& \sum_{j=1}^q\overline{N}_{f}(r, a_j)+O\big{(}\log^+ T_{f}(r)+ r\big{)}
  \end{eqnarray*}
holds for all $r\in(0,\infty)$ outside a set $E_\delta$ of finite Lebesgue measure. 
\end{cor}

\noindent\textbf{Remark.}  In fact, Corollary \ref{ccc1} is equivalent to the  case  where  $m=1$  and $N=\mathbb P^1(\mathbb C)$ in Corollary \ref{ccc}, i.e., the following Second Main Theorem  
$$      (q-2)T_{f}(r)
     \leq \sum_{j=1}^q\overline{N}_{f}(r, a_j)+O\Big{(}\log^+ T_{f}(r)+\log\frac{1}{1-r}\Big{)}.$$ 
To see this equivalence, we just need to compare the Second Main Theorem for  $\mathbb D$ under the  Poincar\'e metric and  Euclidean metric. 
To avoid confusion, denote by $r, \tilde r$  the geodesic radius under  the Poincar\'e metric and   
Euclidean metric respectively,   by $r(x), \tilde r(x)$ the Riemannian distance functions under  the Poincar\'e metric and   
Euclidean metric respectively.  Similarly,   denote by $g_r(o, x), \tilde g_{\tilde r}(o, x)$ as well  as $T_f(r), \tilde T_f(\tilde r)$  the Green functions and characteristic functions under the metrics.

Firstly, we  compare the main error terms, i.e.,  $O(r)$ and $O(-\log(1-\tilde r)).$
Take  $o$ as the coordinate origin of $\mathbb D$.  By the relation 
\begin{equation}\label{ppp1}
r(x)=\log\frac{1+\tilde r(x)}{1-\tilde r(x)},
\end{equation}
we see that  $r$ corresponds to 
$$\log \frac{1+\tilde r}{1-\tilde r}=\log\frac{1}{1-\tilde r}+O(1)$$
due to $\tilde r<1.$ Thus, the two error terms are equivalent.

Finally, we compare the characteristic functions, i.e., $T_f(r)$ and $\tilde T_f(\tilde r).$ The similar discussions can be applied to the comparisons for 
counting functions and proximity functions. Under the Euclidean metric,  the Green function is written as
$$\tilde g_{\tilde r}(o, x)=\frac{1}{\pi}\log\frac{\tilde r}{|x|}=\frac{1}{\pi}\log\frac{\tilde r}{\tilde r(x)}.$$
which corresponds to the Green function $g_r(o,x)$ under the   Poincar\'e metric since (\ref{qqq3}) and (\ref{ppp1}). Notice that
$$\Delta\log(1+|f|^2)dV=\tilde\Delta\log(1+|f|^2)d\tilde V,$$
where $\Delta, \tilde\Delta$ denote   Laplace-Beltrami operators under the  Poincar\'e metric and Euclidean metric respectively,  and $dV, d\tilde V$ denote  
volume elements under the  Poincar\'e metric and Euclidean metric respectively. By the definition of characteristic function, we see that they are a match.
Hence,  the two Second Main Theorems are actually equivalent under the two  metrics. 

\noindent{\textbf{2. Some other  consequences}}

 Let $M$ be a Riemannian manifold with a point  $o\in M.$ We establish a polar coordinate system $(o, r, \theta).$  For any $x=(r, \theta)\in M$ such that $x\not\in Cut^*(o),$ denote by $Ric_o(x)$ the Ricci curvature of $M$ at $x$ in the direction $\partial/\partial r.$ Let $\omega=(\partial/\partial r, X)$ be any pair of tangent vectors from $T_xM,$ where $X$ is a unit vector orthogonal to $\partial/\partial r.$ Indeed, let $K_\omega(x)$ be  the sectional curvature of $M$ at $x$ along the 2-section determined by $\omega.$
For a $d$-dimensional spherically symmetric  manifold $M_\sigma,$ a direct computation \cite{bi} yields that 
\begin{equation}\label{sss}
Ric_o(x)=-(d-1)\frac{\sigma''(r)}{\sigma(r)}, \ \ \ K_\omega(x)=-\frac{\sigma''(r)}{\sigma(r)}
\end{equation}
for all $x=(r, \theta)\in M_\sigma\setminus o.$

\begin{lemma}[Ichihara, \cite{IK, I2}]\label{lem1}  Let $\psi$ be a smooth positive function on $(0, \infty)$ such that 
$$\psi(0)=0, \ \ \ \psi'(0)=1.$$
Let $M$ be a $d$-dimensional geodesically complete, non-compact manifold, and $o\in M.$ Set 
$$S(r)=\omega_{d-1}\psi^{d-1}(r),$$
 where $\omega_{d-1}$ is the area of $S^{d-1}.$ Then
 
 $(a)$ If for all $x=(r, \theta)\not\in Cut^*(o)$ 
 $$Ric_o(x)\geq-(d-1)\frac{\psi''(r)}{\psi(r)}, \ \ \ \int_1^\infty\frac{dr}{S(r)}=\infty,$$
 then $M$ is parabolic. 
 
 $(b)$ If for all $x=(r, \theta)\not=o$ and all $\omega$
 $$K_\omega(x)\geq-\frac{\psi''(r)}{\psi(r)}, \ \ \ \int_1^\infty\frac{dr}{S(r)}<\infty,$$
 then $M$ is non-parabolic.
\end{lemma}

\begin{cor}\label{cor1} Let $M_\sigma$ be a  $d$-dimensional geodesically complete and non-compact spherically symmetric manifold. Then $M_\sigma$ is parabolic if and only if
$$\int_1^\infty\frac{dr}{\sigma^{d-1}(r)}=\infty.$$
\end{cor} 
\begin{proof}
The conclusion follows immediately from (\ref{sss}) and Lemma \ref{lem1}.
\end{proof}

By Corollary \ref{cor1},  $T_{f}(r,L)\rightarrow\infty$ as $r\rightarrow\infty$ if $M_\sigma$ is parabolic.
\begin{theorem}\label{thm1}   
 Let $M_\sigma$ be a geodesically complete and non-compact spherically symmetric  K\"ahler  manifold of complex dimension $m$. Let 
  $L$ be a positive line bundle over a complex projective manifold $N$ with $\dim_{\mathbb C}N\leq m,$  and  $D\in|L|$  be  of  simple normal crossings.
 Let $f: M_\sigma\rightarrow N$ be a differentiably non-degenerate holomorphic map. Assume that $M_\sigma$ is parabolic. Then for every $\delta>0$  
  \begin{eqnarray*}
    &&  T_{f}(r,L)+T_{f}(r,K_N)+T(r,\mathscr R_{M_\sigma}) \\
     &\leq& \overline{N}_{f}(r,D)+O\big{(}\log^+ T_{f}(r,L)+\delta\log^+r\big{)}
  \end{eqnarray*}
holds for all $r\in(0,\infty)$ outside a set $E_\delta$ of finite Lebesgue measure.
\end{theorem}
\begin{proof} The completeness and non-compactness imply that $M_\sigma$ have radius $R=\infty.$ By Corollary \ref{cor1}, the parabolicity of $M_\sigma$ implies that 
$$\int_1^\infty\frac{dr}{\sigma^{2m-1}(r)}=\infty,$$
which leads to $\log^+\sigma(r)\leq O(\log^+r).$  Apply Theorem \ref{main theorem} (a), we prove the theorem. 
\end{proof}

\begin{cor}\label{}    
Let $M_\sigma$ be a geodesically complete and non-compact spherically symmetric  K\"ahler  manifold of complex dimension $m$.  Let 
  $L$ be a positive line bundle over a complex projective manifold $N$ with $\dim_{\mathbb C}N\leq m,$  and  $D\in|L|$  be of  simple normal crossings.
 Let $f: M_\sigma\rightarrow N$ be a differentiably non-degenerate holomorphic map. Assume that the Ricci curvature of  $M_\sigma$ as a Riemannian manifold is non-negative. Then for every $\delta>0$  
  \begin{eqnarray*}
      T_{f}(r,L)+T_{f}(r,K_N)
     &\leq& \overline{N}_{f}(r,D)+O\big{(}\log^+ T_{f}(r,L)+\delta\log^+r\big{)}
  \end{eqnarray*}
holds for all $r\in(0,\infty)$ outside a set $E_\delta$ of finite Lebesgue measure.
\end{cor}
\begin{proof} The non-negativity of  Ricci curvature  implies the parabolicity of $M_\sigma,$  hence the conclusion holds by using Theorem \ref{thm1}. 
\end{proof}
\textbf{Proof of Theorem \ref{main theorem}}

\begin{proof}
Write
 $D=\sum_{j=1}^qD_j$ as the union of irreducible components and equip every $L_{D_j}$ with a Hermitian
metric $h_j.$ Then it induces a natural Hermitian metric $h_L=h_1\otimes\cdots\otimes h_q$ on $L,$  
 which defines  a  volume form $\Omega:=c^n_1(L,h_L)$ on $N.$
Pick $s_j\in H^0(N, L_{D_j})$
with $D_j=(s_j)$ and $\|s_j\|<1.$
On $N$, one defines a singular volume form
$$
  \Phi=\frac{\Omega}{\prod_{j=1}^q\|s_j\|^2}.
$$Set
$$\xi\alpha^m=f^*\Phi\wedge\alpha^{m-n}, \ \ \alpha=\frac{\sqrt{-1}}{\pi}\sum_{i,j=1}^mh_{i\bar j}dz_i\wedge d\bar z_j.$$
Note that
$$\alpha^m=m!\det(h_{i\bar j})\bigwedge_{j=1}^m\frac{\sqrt{-1}}{\pi}dz_j\wedge d\bar z_j.$$
A direct computation leads to
$$dd^c\log\xi\geq f^*c_1(L, h_L)-f^*{\rm{Ric}}(\Omega)+\mathscr{R}_{M_\sigma}-{\rm{Supp}}f^*D$$
in the sense of currents, where $\mathscr{R}_{M_\sigma}=-dd^c\log\det(h_{i\bar j}).$
This follows that
\begin{eqnarray}\label{5q}
&& T(r, dd^c\log\xi) \\
&\geq& T_{f}(r,L)+T_{f}(r,K_N)+T(r,\mathscr{R}_{M_\sigma})-\overline{N}_{f}(r,D)+O(1). \nonumber
\end{eqnarray}

Next we give an upper bound of $ T(r, dd^c\log\xi).$ The simple normal crossing  property of $D$ implies that
there exist a finite  open covering $\{U_\lambda\}$ of $N$ and finitely many  rational functions
$w_{\lambda1},\cdots,w_{\lambda n}$ on $N$  such that $w_{\lambda1},\cdots, w_{\lambda n}$ are holomorphic on $U_\lambda$ for each $\lambda$ as well as
\begin{eqnarray*}
  dw_{\lambda1}\wedge\cdots\wedge dw_{\lambda n}(y)\neq0, & & \ ^\forall y\in U_{\lambda}, \\
  D\cap U_{\lambda}=\big{\{}w_{\lambda1}\cdots w_{\lambda h_\lambda}=0\big{\}}, && \ ^\exists h_{\lambda}\leq n.
\end{eqnarray*}
Indeed, we  can require  that $L_{D_j}|_{U_\lambda}\cong U_\lambda\times \mathbb C$ for 
$\lambda,j.$ On  $U_\lambda,$   write 
$$\Phi=\frac{e_\lambda}{|w_{\lambda1}|^2\cdots|w_{\lambda h_{\lambda}}|^2}
\bigwedge_{k=1}^n\frac{\sqrt{-1}}{2\pi}dw_{\lambda k}\wedge d\bar w_{\lambda k},$$
where  $e_\lambda$ is a  positive smooth function on $U_\lambda.$  
 Set 
$$\Phi_\lambda=\frac{\phi_\lambda e_\lambda}{|w_{\lambda1}|^2\cdots|w_{\lambda h_{\lambda}}|^2}
\bigwedge_{k=1}^n\frac{\sqrt{-1}}{2\pi}dw_{\lambda k}\wedge d\bar w_{\lambda k},$$
where  $\{\phi_\lambda\}$ is a partition of unity subordinate to $\{U_\lambda\}.$ 
Let $f_{\lambda k}=w_{\lambda k}\circ f$, then on  $f^{-1}(U_\lambda)$ 
\begin{eqnarray*}
 f^*\Phi_\lambda&=&
   \frac{\phi_{\lambda}\circ f\cdot e_\lambda\circ f}{|f_{\lambda1}|^2\cdots|f_{\lambda h_{\lambda}}|^2}
   \bigwedge_{k=1}^n\frac{\sqrt{-1}}{2\pi}df_{\lambda k}\wedge d\bar f_{\lambda k} \\
   &=& \phi_{\lambda}\circ f\cdot e_\lambda\circ f\sum_{1\leq i_1\not=\cdots\not= i_n\leq m}
   \frac{\Big|\frac{\partial f_{\lambda1}}{\partial z_{i_1}}\Big|^2}{|f_{\lambda 1}|^2}\cdots 
   \frac{\Big|\frac{\partial f_{\lambda h_\lambda}}{\partial z_{i_{h_\lambda}}}\Big|^2}{|f_{\lambda h_\lambda}|^2}
   \Big|\frac{\partial f_{\lambda (h_\lambda+1)}}{\partial z_{i_{h_\lambda+1}}}\Big|^2\cdots\Big|\frac{\partial f_{\lambda n}}{\partial z_{i_{n}}}\Big|^2 \\
   && \cdot\Big(\frac{\sqrt{-1}}{2\pi}\Big)^ndz_{i_1}\wedge d\bar z_{i_1}\wedge\cdots\wedge dz_{i_n}\wedge d\bar z_{i_n}.
\end{eqnarray*}
\ \ \ \ Fix any  $x_0\in M_\sigma,$ we may pick  a holomorphic coordinate system $z_1,\cdots,z_m$ near $x_0$ and  a holomorphic coordinate system $w_1,\cdots,w_m$ 
near $f(x_0)$ so that
$$\alpha=\frac{\sqrt{-1}}{2\pi}\sum_{j=1}^m dz_j\wedge d\bar{z}_j, \ \ \ c_1(L, h_L)|_{f(x_0)}=\frac{\sqrt{-1}}{2\pi}\sum_{j=1}^n dw_j\wedge d\bar{w}_j.
$$
 Set
$$f^*\Phi_\lambda\wedge\alpha^{m-n}=\xi_\lambda\alpha^m.$$
Then we have $\xi=\sum_\lambda \xi_\lambda$ and  
 \begin{eqnarray*}
\xi_\lambda|_{x_0} &=& \phi_{\lambda}\circ f\cdot e_\lambda\circ f\sum_{1\leq i_1\not=\cdots\not= i_n\leq m}
   \frac{\Big|\frac{\partial f_{\lambda1}}{\partial z_{i_1}}\Big|^2}{|f_{\lambda 1}|^2}\cdots 
   \frac{\Big|\frac{\partial f_{\lambda h_\lambda}}{\partial z_{i_{h_\lambda}}}\Big|^2}{|f_{\lambda h_\lambda}|^2}
   \Big|\frac{\partial f_{\lambda (h_\lambda+1)}}{\partial z_{i_{h_\lambda+1}}}\Big|^2\cdots\Big|\frac{\partial f_{\lambda n}}{\partial z_{i_{n}}}\Big|^2 \\
   &\leq&  \phi_{\lambda}\circ f\cdot e_\lambda\circ f\sum_{1\leq i_1\not=\cdots\not= i_n\leq m}
    \frac{\big\|\nabla_{M_\sigma}f_{\lambda1}\big\|^2}{|f_{\lambda 1}|^2}\cdots 
   \frac{\big\|\nabla_{M_\sigma}f_{\lambda h_\lambda}\big\|^2}{|f_{\lambda h_\lambda}|^2} \\
   &&\cdot \big\|\nabla_{M_\sigma}f_{\lambda(h_\lambda+1)}\big\|^2\cdots\big\|\nabla_{M_\sigma}f_{\lambda n}\big\|^2.
\end{eqnarray*}
Again, set
\begin{equation}\label{gtou}
  f^*c_1(L, h_L)\wedge\alpha^{m-1}=\varrho\alpha^m.
\end{equation}
Let $f_j=w_j\circ f$ for $1\leq j\leq n,$  then 
\begin{equation*}
    f^*c_1(L, h_L)\wedge\alpha^{m-1}|_{x_0}=\frac{(m-1)!}{2}\sum_{j=1}^m\big\|\nabla_{M_\sigma} f_j\big\|^2\alpha^m.
\end{equation*}
That is,
$$ \varrho|_{x_0}=(m-1)!\sum_{i=1}^n\sum_{j=1}^m\Big|\frac{\partial f_i}{\partial z_j}\Big|^2
=\frac{(m-1)!}{2}\sum_{j=1}^n\big\|\nabla_{M_\sigma} f_j\big\|^2.$$
Combine the above, we are led to 
$$\xi_\lambda\leq 
\frac{ \phi_{\lambda}\circ f\cdot e_\lambda\circ f\cdot(2\varrho)^{n-h_\lambda}}{(m-1)!^{n-h_\lambda}}\sum_{1\leq i_1\not=\cdots\not= i_n\leq m}
    \frac{\big\|\nabla_{M_\sigma}f_{\lambda1}\big\|^2}{|f_{\lambda 1}|^2}\cdots 
   \frac{\big\|\nabla_{M_\sigma}f_{\lambda h_\lambda}\big\|^2}{|f_{\lambda h_\lambda}|^2}
$$
on $f^{-1}(U_\lambda).$
Note that $\phi_\lambda\circ f\cdot e_\lambda\circ f$ is bounded on $M_\sigma,$ then it follows from $\log\xi\leq \sum_\lambda\log^+\xi_\lambda+O(1)$ that 
\begin{equation}\label{bbd}
   \log^+\xi\leq O\Big{(}\log^+\varrho+\sum_{k, \lambda}\log^+\frac{\|\nabla_{M_\sigma} f_{\lambda k}\|}{|f_{\lambda k}|}\Big{)}+O(1) 
 \end{equation}  on $M_\sigma.$ By Dynkin formula (see Section 3.1.2)
\begin{equation}\label{pfirst}
T(r,dd^c\log\xi)
=\frac{1}{2}\int_{S_o(r)}\log\xi(x)d\pi^r_o(x)+O(1).
\end{equation}
By  (\ref{bbd}) and (\ref{pfirst}) with Theorem \ref{LDL} 
\begin{eqnarray*}
    T(r, dd^c\log\xi) 
   &\leq& O\Big{(}\sum_{k,\lambda}m\Big{(}r,\frac{\|\nabla_{M_\sigma} f_{\lambda k}\|}{|f_{\lambda k}|}\Big{)}+\log^+\int_{S_o(r)}\varrho(x)d\pi^r_o(x)\Big)+O(1) \\
   &\leq& O\Big{(}\sum_{k,\lambda}\log^+ T(r,f_{\lambda k})+\log^+\int_{S_o(r)}\varrho(x)d\pi^r_o(x)\Big{)}+O(1)
   \\
   &\leq& O\Big{(}\log^+ T_{f}(r,L)+\log^+\int_{S_o(r)}\varrho(x)d\pi^r_o(x)\Big{)}+O(1).   
   \end{eqnarray*}
 Lemma \ref{CL} and (\ref{gtou}) imply that for every $\delta>0$
\begin{eqnarray*}
&&\log^+\int_{S_o(r)}\varrho(x)d\pi^r_o(x) \\
   &\leq& (1+\delta)^2\log^+T_{f}(r,L)+(2+\delta)\log^+\gamma(r)+(2m-1)\delta\log^+\sigma(r)
   \end{eqnarray*}
   holds for all $r\in(0,R)$ outside a set $E_\delta$ with $\int_{E_{\delta}}\gamma(r)dr<\infty.$
Thus, 
\begin{equation}\label{6q}
    T(r,dd^c\log\xi) \leq O\big{(}\log^+ T_{f}(r,L)+\log^+\gamma(r)+\delta\log^+\sigma(r)\big{)}+O(1)
\end{equation}
  for all $r\in(0,R)$ outside  $E_\delta$ with $\int_{E_{\delta}}\gamma(r)dr<\infty.$  Combining (\ref{5q}) with (\ref{6q}),  we prove the theorem.
\end{proof}

\subsection{Defect relations}~ 
 
Recall the definition of simple defect $\Theta_f(D)$ in Introduction.  
 \begin{theorem}\label{} Assume the same conditions as in Theorem $\ref{main theorem}.$   
 
$(a)$ For $R=\infty,$ if $T_{f}(r,L)\geq O(\log^+\sigma(r))$ as $r\rightarrow \infty,$ then 
$$\Theta_f(D)\leq \left[\frac{c_1(K^*_N)}{c_1(L)}\right]-\liminf_{r\rightarrow\infty}\frac{T(r,\mathscr R_{M_\sigma}) }{T_{f}(r,L)}.$$

$(b)$ For $R<\infty,$   if  $\log\frac{1}{R-r}=o(T_{f}(r,L))$ as $r\rightarrow R,$ then 
 $$\Theta_f(D)\leq \left[\frac{c_1(K^*_N)}{c_1(L)}\right]-\liminf_{r\rightarrow R}\frac{T(r,\mathscr R_{M_\sigma})}{T_{f}(r,L)}.$$
\end{theorem}
 \begin{proof}
 The conclusions follow  directly from Theorem $\ref{main theorem}.$
 \end{proof}
\begin{cor}\label{defect} Assume the same conditions as in Theorem $\ref{main theorem},$  Suppose also that   $M_\sigma$ has non-negative scalar curvature. 

$(a)$ For $R=\infty,$ if $T_{f}(r,L)\geq O(\log^+\sigma(r))$ as $r\rightarrow \infty,$ then 
$$\Theta_f(D)\leq \left[\frac{c_1(K^*_N)}{c_1(L)}\right].$$

$(b)$ For $R<\infty,$   if  $\log(R-r)=o(T_{f}(r,L))$ as $r\rightarrow R,$ then 
 $$\Theta_f(D)\leq \left[\frac{c_1(K^*_N)}{c_1(L)}\right].$$
\end{cor}
 \begin{proof}
 The non-negativity of scalar curvature of $M_\sigma$ gives that $T(r,\mathscr R_{M_\sigma})\geq0,$ see (\ref{scalar}).
\end{proof}

\vskip\baselineskip

\label{lastpage-01}
\end{document}